\documentclass[11pt]{amsart}

\usepackage{amssymb,verbatim}
\usepackage{amsmath,amsfonts}
\usepackage[mathscr]{euscript} 
\usepackage{amsthm}
\usepackage{url}
\usepackage{graphicx} 
\usepackage{enumitem} 
\usepackage{hyperref} 
\hypersetup{colorlinks} 
\usepackage{bbm} 
\usepackage{bm} 
\usepackage{cancel} 

\usepackage[colorinlistoftodos]{todonotes}

\RequirePackage{cleveref}
\usepackage{hypcap}
\hypersetup{colorlinks=true, citecolor=darkblue, linkcolor=darkblue}
\definecolor{darkblue}{rgb}{0.0,0,0.7}
\newcommand{\darkblue}{\color{darkblue}}

\definecolor{darkred}{rgb}{0.68,0,0}

\definecolor{darkgreen}{rgb}{0,.38,0}

\newcommand{\defn}[1]{\emph{\darkblue #1}}

\newcommand{\defnb}[1]{\emph{\darkblue #1}}

\setlist[enumerate]{
	label=\textnormal{({\roman*})},
	ref={\roman*}}

\makeatletter
\def\th@plain{%
	\thm@notefont{}
	\itshape 
}
\def\th@definition{%
	\thm@notefont{}
	\normalfont 
}
\makeatother

\newtheorem{thm}{Theorem}[section]
\newtheorem{lemma}[thm]{Lemma}

\newtheorem*{claim*}{Claim}

\newtheorem{prop}[thm]{Proposition}

\newtheorem{conv}[thm]{Convention}

\theoremstyle{definition}

\newtheorem{rem}[thm]{Remark}

\numberwithin{figure}{section}
\numberwithin{equation}{section}

\def\nn{\mathbb N}

\def\qqq{\mathbb Q}

\def\al{\alpha}

\def\<{\langle}
\def\>{\rangle}

\def\rp{ {\text {\rm p}  } }

\def\SL{ {\text {\rm SL} } }

\def\0{{\mathbf 0}}

\def\.{\hskip.06cm}
\def\ts{\hskip.03cm}

\def\br{{\ze}}

\def\ze{{\zeta}}

\def\.{\hskip.06cm}
\def\ts{\hskip.03cm}

\def\aS{\textrm{S}}
\def\aSr{\textrm{\em S}}

\DeclareMathOperator{\ac}{{\text{$c$}}}
\DeclareMathOperator{\ad}{{\text{$d$}}}
\DeclareMathOperator{\at}{{\text{$t$}}}

\DeclareMathOperator{\bI}{\rho}

\usetikzlibrary{positioning,shapes.geometric,calc}

\title
[Effective resistance in planar graphs]
{Effective resistance in planar graphs \\ and continued fractions }
\date{\today}

 \author{Swee Hong Chan}
 \address[Swee Hong Chan]{Department of Mathematics, Rutgers University,  Piscataway, NJ 08854.}
 \email{\texttt{sweehong.chan@rutgers.edu}}

 \author{Alex Kontorovich}
 \address[Alex Kontorovich]{Department of Mathematics, Rutgers University,  Piscataway, NJ 08854.}
 \email{\texttt{alex.kontorovich@rutgers.edu}}

 \author[\ts Igor Pak]{Igor Pak}
 \address[Igor Pak]{Department of Mathematics, UCLA,  Los Angeles, CA 90095.}
 \email{\texttt{pak@math.ucla.edu}}

\begin{document}

\begin{abstract}
For a simple graph $G=(V,E)$ and edge $e\in E$, the
\emph{effective resistance} is defined as a ratio \ts
$\frac{\tau(G/e)}{\tau(G)}\ts$, where $\tau(G)$ denotes
the number of spanning trees in~$G$.  We resolve the \emph{inverse
problem} for the effective resistance for planar graphs.
Namely, we determine (up to a constant) the smallest size
of a simple planar graph with a given effective resistance.  The results
are motivated and closely related to our previous work \cite{CKP}
on Sedl\'a\v{c}ek's inverse problem for the number of spanning trees.
\end{abstract}
	
\maketitle

\section{Introduction}\label{s:intro}

\smallskip

\subsection{Spanning trees and effective resistance}
Let \ts $G=(V,E)$ \ts be a connected graph without loops.
For an edge \ts $e \in E$,
denote by \ts $G/e$ \ts the graph obtained by {contracting}
the edge~$e$, where all resulting loops are removed.

Let \ts $\tau(G)$ \ts denote the \emph{number of spanning trees} of~$G$.
The \defn{effective resistance} is defined as
\begin{equation}\label{eq:bI}
		\bI(G,e) \ := \  \frac{\tau(G/e)}{\tau(G)}\..
\end{equation}

This notion goes back to Kirchhoff (1847) in the context of
\emph{electrical networks}, as it measures the current
through the edge $e=(x,y)$, as a fraction of the current between
nodes
$x$ and~$y$.
Effective resistance is one of the main
graph
invariants, with applications across mathematics and the sciences.
Notably, it was proved by Nash-Williams~\cite{Nas}, that \ts $\bI(G,e)$ \ts
is the probability that a simple random walk which
starts
at $x$ and
exits
at~$y$, traverses~$e$ at some point.
We refer to \cite{DS,Lov-survey,LP} for
background,
modern proofs,
and further references.

\smallskip


Clearly, 
effective resistance can be computed in poly-time via the \emph{matrix tree theorem}.
In this paper, we study the \emph{inverse problem}, of finding the smallest size graph with a
given effective resistance.  We restrict ourselves to planar graphs.
%
%

 \smallskip

 \begin{thm}[{\rm Main theorem}{}]\label{t:main}
Let \ts $\at>\ac\ge 1$ \ts be coprime integers. Then there exists a
simple planar graph \ts $G=(V,E)$ \ts and an  edge \ts $e\in E$,  such that
\begin{equation}\label{eq:ratio}
 \bI(G,e) \, = \,  \frac{\ac}{\at}
\end{equation}
and
\begin{align}
  \label{eq:Kir-vertex} |V| \ \leq \ C \. \max \bigg \{\. \frac{\at}{\ac} \., \. \frac{\at}{\at-\ac}\., \. \log \at  \bigg\},
\end{align}
for some universal constant $C>0$.
 \end{thm}

 \smallskip

In fact, the upper bound in the theorem is optimal in the following sense:

\smallskip

 \begin{prop}\label{p:main-lower}
Let \ts $G=(V,E)$ \ts be connected simple planar graph,
and let \ts $e\in E$ \ts be an edge that is not a bridge.  Suppose that
\begin{equation}\label{eq:ratio-prop}
 \bI(G,e) \, = \,  \frac{\ac}{\at}
\end{equation}
is a reduced fraction.  Then:
\begin{align}
  \label{eq:Kir-vertex-prop} |V| \ \geq \ C' \. \max \bigg \{\. \frac{\at}{\ac} \., \. \frac{\at}{\at-\ac}\., \. \log \at  \bigg\},
\end{align}
for some universal constant $C'>0$.
 \end{prop}

Here the non-bridge condition is necessary to ensure that \ts $\bI(G,e)<1$.
Proposition~\ref{p:main-lower} follows easily from the results in the literature;
we present a short proof in~$\S$\ref{ss:lower}.  See~$\S$\ref{ss:finrem-const}
for an explicit value of the constant~$C'$.

\smallskip

The idea of inverse problems for combinatorial functions was developed independently
in several different areas.  Recently, it reemerged in the computational complexity
context and systematically studied by the first and third author \cite{CP-coinc}.
This paper is a continuation of our previous study \cite{CKP} of
\emph{Sedl\'a\v{c}ek's problem}, which is the inverse problem for the
number of spanning trees in simple graphs (see below).  There, we introduced a new
approach based on
continued
 fractions, and used the \emph{Bourgain--Kontorovich
technology} \cite{BK14} to obtain the result.  Here we follow a closely related
approach.


Theorem~\ref{t:main} extends the earlier result \cite[Lemma~1.14]{CP-SY},
where a  weaker bound \. $|V| = O(\log \at \. (\log \log \at)^2)$ \ts was
obtained for \. $\tfrac{\ac}{\at}\in \big [\tfrac{1}{3},\tfrac{2}{3}\big]$.
Additionally, the graphs in \cite{CP-SY} were allowed to have multiple edges.

\smallskip

\subsection{Sedl\'a\v{c}ek's problem}\label{ss:intro-Sedlacek}
Given a positive integer \ts $t\ge 3$, let \ts $\alpha(\at)$ \ts be the smallest
number vertices of a simple graph with \emph{exactly} \ts $\at$ \ts spanning trees.
The study of the function \ts $\alpha(\at)$ \ts was initiated by Sedl\'a\v{c}ek
in a series of papers
\cite{Sed66,Sed69,Sed70}, and remains unresolved.  The best known upper bound is
\. $\alpha(\at)=O((\log \at)^{3/2} / (\log \log \at))$ \. due to  Stong~\cite{Stong},
while best known lower bound is \.
$\alpha(\at) = \Omega(\log \at/\log \log \at)$,  which
follows from \emph{Cayley's formula} \. $\tau(K_n)=n^{n-2}$.
Azarija and \u{S}krekovski \cite{AS12} conjectured that \ts $\alpha(t)=o(\log \at)$.

In \cite{CKP}, the authors consider another, closely related problem by Sedl\'a\v{c}ek,
on the function \ts $\al_{\rp}(\at)$ \ts defined as the smallest
number vertices of a simple \emph{planar} graph with exactly \ts $\at$ \ts
spanning trees.  Clearly, we have \ts $\al(\at) \le \al_{\rp}(\at)$ \ts for all~$t$.  
It is known
and easy to see that \ts $\al_{\rp}(\at) = \Omega(\log \at)$, see~$\S$\ref{ss:lower}.
The main result in \cite{CKP} is that \ts $\al_{\rp}(\at) = O(\log \at)$ \ts for a
set of integers \ts $\at$ \ts of positive density.  One can think of Theorem~\ref{t:main}
as a tradeoff: the density condition is removed but the number of spanning trees
is replaced with a ratio of two such numbers.

In a different direction, it follows from Theorem~\ref{t:main} that for all integers \ts
$\at\ge 2$, there exists a connected simple planar graph \ts $G=(V,E)$ \ts such that
\begin{equation}\label{eq:k-Sedlacek}
	   \tau(G) \, \equiv \, 0 \mod \at \quad \text{ and } \quad |V| \, = \, O(\log \at).
\end{equation}
This is also a special case of the following modular version of Theorem~\ref{t:main}.

\smallskip

\begin{thm}[{\rm Alon--Buci\'c--Gishboliner \cite[Thm~3.1]{ABG}}{}]\label{thm:ABG}
	Let \ts $\at\ge 2$ \ts be a positive integer, and let \ts $a,b\in \nn$ \ts
be such that \. $(a,b)\neq (0,0) \ \, \textnormal{mod }   \at$.
	Then there exists a simple planar graph \ts $G=(V,E)$ \ts and an edge \ts $e\in E$, such that
	\[ \tau(G/e) \, \equiv \, a \mod \at, \quad \tau(G) \, \equiv \, b \mod \at, \quad \text{ and } \quad |V| \, = \,  O(\log \at).   \]
\end{thm}

\smallskip

The equation \eqref{eq:k-Sedlacek} follows from Theorem~\ref{thm:ABG} by setting $a=1$ and $b=0$.
The proof in \cite{ABG} involves the celebrated expander construction of \ts $\SL(2,p)$
\ts based on Selberg's theorem \cite{Sel65}.  Curiously, the same Selberg's theorem is the
starting point of a long chain of results leading to \cite{Bou12,BK14}, which provided main
tools for both \cite{CKP} and for this paper.

 \medskip

 \subsection{Connections to continued fractions}\label{ss:intro-CF}
%
%
 Given the integers \. $a_0\geq 0$ \. and \. $a_1, \ldots, a_\ell \geq 1 $, 
 the corresponding \defnb{continued fraction} \ts is defined as
 \[ [a_0\ts ; \ts a_1,\ldots, a_\ell] \ := \ a_0  \. +  \.  \cfrac{1}{a_1  \. +  \. \cfrac{1}{\ddots  \ +  \.  \frac{1}{a_\ell}}}
 \]
 Integers \ts $a_i$ \ts are called \defn{partial quotients}, see e.g.\ \cite[$\S$10.1]{HardyWright}.
 We use the notation \ts $[a_1,\ldots, a_\ell]$ \ts when \ts $a_0=0$.
 Every positive rational number \. $q$ \. admits  two  continued fraction representations  since \. $[a_0;a_1,\ldots, a_\ell,1]=[a_0;a_1,\ldots, a_\ell+1]$\..
 For $q \in \qqq_{>0}$, we define
  $$
  \aS(q) \, :=  \, a_0 \. + \. a_1 \. + \. \ldots\. + \. a_\ell
  $$
 and note that \ts $\aS(q)$ \ts is the same for both representations.
 We also define \. $\aS(0):=0$.



We need the following remarkable result which was proved
as a consequence of the {Bourgain--Kontorovich technology}~\cite{BK14}.

 \smallskip

 \begin{thm}[{\rm Bourgain \cite[Prop.~1]{Bou12}}{}]\label{thm:Bou}
 	 Every rational number \. $\tfrac{\ad}{\ac} \in [0,1)$ \. can
be written as the following sum:
\begin{align}\label{eq:Bou}
  \frac{\ad}{\ac} \, = \, q_1 \. + \. \ldots \. + \. q_k\.,
\end{align}
where \. $q_1,\ldots, q_k\in \qqq \cap (-1,1)$ \. satisfy
 \[ \aSr\big(|q_1|\big)  \. + \. \ldots \. + \. \aSr\big(|q_k|\big) \ \leq \ C \ts \log (\ac+\ad),  \]
for some universal constant \ts $C>0$.
 \end{thm}

 \smallskip

We apply Theorem~\ref{thm:Bou} to  our problem via the following theorem,
which converts the problem of  analyzing  ratios of spanning trees
into a problem on continued fractions.  For an edge \ts $e \in E$,
we denote by \ts $G-e$ \ts the graph obtained by \defnb{deleting} the edge $e$ on $G$.
 \smallskip

\begin{thm}\label{thm:graph-sum}
 Let \. $q_1,\ldots, q_k\in \qqq_{>0}$ \. be  positive rational numbers.
 Then there exists a simple planar graph \ts $G=(V,E)$ \ts and an edge \ts
 $e\in E$,  such that
 	\[
        \frac{\tau(G-e)}{\tau(G/e)} \ = \  q_1 \. + \. \ldots \. + \. q_k
    \]
 	and
 	\[ |E| \, = \,  4\big(\aSr(q_1) \. + \. \ldots \. + \. \aSr(q_k)\big) \. + \. 1. \]
 \end{thm}

\smallskip

Note that Theorem~\ref{thm:Bou} and Theorem~\ref{thm:graph-sum} immediately imply Theorem~\ref{t:main} when  all rational numbers \.  $q_1,\ldots, q_k$ \. in the representation of  \. $\tfrac{\at-\ac}{\ac}$ \. in \eqref{eq:Bou} are  nonnegative.
The full version of Theorem~\ref{t:main} can then be obtained through a more involved argument that builds on both Theorem~\ref{thm:Bou} and Theorem~\ref{thm:graph-sum}. 

\begin{rem}
We note that Theorem~\ref{thm:Bou} cannot be strengthened to require that all \ts
$q_i$ \ts 
are nonnegative.  Indeed,
if \ts $q\in\qqq_{>0}$ \ts is less than \ts $1/c$, then \ts $S(q)\ge c$.  
Therefore, any expression of the unit fraction \ts $1/c$ \ts as \ts 
$1/c = q_1+\ldots+q_k$ \ts with positive $q_i$ must have \ts $S(q_1)+\ldots +S(q_k)\ge kc$.
\end{rem}

 \medskip

\section{Preliminaries}\label{s:prelim}

In this section, we collect several graph theoretic lemmas and constructions
that will be used in the proofs in Section~\ref{s:proof-main}.  Several of
these constructions are standard and have previously appeared in \cite{CP-SY, CKP},
but
our
\emph{simplification} \ts is new or at least used in a nonstandard way.

\subsection{Basic definitions}\label{ss:prelim-basic}
Throughout this paper,
 all graphs are assumed to be planar and have no loops.
Graphs \emph{are} allowed to have multiple edges unless indicated otherwise.
A graph
with no loops
is \defnb{simple} if it
also
does not have multiple edges.
Let \ts $G=(V,E)$ \ts be a  graph and let \ts $e \in E$ \ts be an edge in~$G$.
A pair \ts $(G,e)$ \ts is called a \defn{marked graph}.
%
Our operations are defined on marked graphs.

\smallskip

\subsection{Subdivisions and duplications}\label{ss:subdivision-duplication}
%
%
%
The \defnb{$k$-subdivision} of \ts $(G,e)$ \ts
 is  a marked graph \ts $(G',e')$ \ts
obtained by replacing edge~$e$ with a path of length $(k+1)$, and marking one of the
new edges as~$e'$. 
Note that if $G$ is connected, simple and planar,
then
so is~$G'$.
When $k=1$,  we omit the parameter and write \defnb{subdivision}.

The \defnb{$k$-duplication} of \ts $(G,e)$ \ts
is  a marked graph \ts $(G'',e'')$ \ts obtained
by replacing edge~$e$ with $(k+1)$ parallel edges, and marking one of the
new edges as~$e''$. 
Note that if $G$ is connected and planar,
then
so is~$G''$.
When $k=1$,  we omit the parameter and write \defnb{duplication}.
Note also that these two operations are planar dual to each other
(see below).



\smallskip

\begin{prop}[{\cite[Thm~5.1]{CP-SY}}]\label{prop:graph-CF}
		Let \ts $\ac,\ad \geq 1$ \ts  be positive integers with  \ts $\gcd(\ac,\ad)=1$.
	Suppose that
	\begin{equation}\label{eq:CFEagain}
		\frac{\ad}{\ac} \ = \   [a_0\ts; \ts a_1,a_2,\ldots, a_\ell],
	\end{equation}
	for some \. $a_0\geq 0,a_1,\ldots, a_{\ell}\ge 1$.
	Then there exists a  planar graph \ts $G=(V,E)$ \ts and an edge \ts $e\in E$ \ts satisfying
	\[
	\tau(G-e) \. = \. \ad,  \qquad \tau(G/e) \. = \. \ac,
	\]
	and such that
	\begin{align*}
		 |E| \. &= \. a_0 \. + \. a_1 \. + \. \ldots \. + \. a_\ell \. + \. 1.
	\end{align*}
\end{prop}

\smallskip

The proof of the proposition is given by a repeated subdivision and duplication
of marked graphs starting with a single edge.  Note that the resulting graph is
not necessarily simple as this required additional constraints on~$a_i\ts$,
see \cite[Lemma~2.3]{CKP}.

\smallskip

\subsection{Marked sum}\label{ss:marked-sum}
A marked graph  \ts $(G,e)$ \ts is called \defnb{proper} \ts if \ts $\tau(G-e)>0$ \ts
and \ts $\tau(G/e)>0$.  In other words, \ts $(G,e)$ is proper if $G$ is connected
and $e$ is not a bridge.
The  \ts \defnb{spanning tree ratio} \ts  of a proper  marked graph is defined as
\[ \br(G,e) \, := \, \frac{\tau(G-e)}{\tau(G/e)} \ > \ 0.  \]
It follows from \eqref{eq:bI} and the \emph{deletion-contraction formula} \. 
$\tau(G-e)+\tau(G/e)=\tau(G)$, that
\begin{equation}\label{eq:ratio-to-current}
	\bI(G,e) \, = \, \frac{1}{1+\br(G,e)}\..
\end{equation}

Let  \ts $(G,e)$ \ts  and \ts $(G',e')$ \ts be marked graphs,
where \ts $G=(V,E)$, \ts $G'=(V',E')$. By definition, we have \ts $e\in E$ \ts and \ts $e'\in E'$.
Define the \defnb{marked sum} \. $(G,e) \ts \oplus \ts (G',e')$ \.
as the marked graph \ts $(G^\circ,e^\circ)$ \ts obtained by  taking the disjoint
union of $G$ and $G'$ and  identifying $e$ with $e'$ as a single edge~$e^\circ$.
Note that if \ts $G$  \ts and \ts $G'$ are proper, planar and simple,
then so is \ts $G^\circ=(V^\circ,E^\circ)$.   Observe that \.
$|V^\circ|  =   |V| + |V'| - 2$ \. and \. $|E^\circ| = |E| + |E'| - 1$.

\smallskip

\begin{lemma}[{\cite[Lem.~5.4]{CP-SY}}]\label{lem:marked-sum}
	For proper marked graphs \ts $(G,e)$ \ts  and \ts $(G',e')$, and
the marked sum \ts $(G^\circ,e^\circ) = (G,e) \oplus (G',e')$ \ts defined above,
we have:
\begin{equation}\label{eq:ratio-marked-sum}
		\br(G^\circ,e^\circ) \, = \,  \br(G,e) \. + \. \br(G',e').
	\end{equation}
\end{lemma}

\smallskip

\subsection{Plane duality}\label{ss:planar-dual}
Let \ts $G=(V,E)$ \ts be a proper planar graph, and let~$F$ be the set
of faces of~$G$.  Denote by \ts $G^\ast=(F,E)$ \ts the \defn{plane dual graph}.
Here we identify edges in~$G$ and~$G^\ast$.
Note that \ts $G^\ast$ \ts is also proper and planar.
By \emph{Euler's formula}, the number of vertices in \ts $G^\ast$ \ts
is given by \ts $|F| = |E|-|V| + 2$.  Recall also that \ts
$\tau(G^\ast) = \tau(G)$.

\smallskip
\begin{lemma}\label{lem:planar-dual}
	For a proper planar marked graph \ts $(G,e)$, we have:
\begin{equation}\label{eq:planar-dual}
\br(G^\ast,e) \, = \,   \frac{1}{\br(G,e)} \quad \text{and} \quad
\bI(G^\ast,e) \, = \, 1 \. - \. \bI(G,e).
\end{equation}	
\end{lemma}

\smallskip

\begin{proof}  Note that deletion and contraction are dual operations.  Thus we have:
\begin{equation}\label{eq:planar-dual-proof}
	 \tau(G-e) \, = \,  \tau(G^*/e^*) \quad \text{and} \quad
	\tau(G/e) \, = \,  \tau(G^*-e^*).
\end{equation}
This and \eqref{eq:ratio-to-current} imply the result.
\end{proof}

\smallskip

\subsection{Simplification}\label{ss:halving}
Let  \ts $G=(V,E)$ \ts be a proper planar simple graph, and let \ts $e\in E$.
Define \defnb{doubling} of the marked graph \ts $(G,e)$ \ts as a marked
graph $(G',e)$, where \ts $G'=(V,E')$ \ts is obtained by duplication of
each edge other than~$e$.  Note that $G'$ is also proper and planar.

Similarly, define \defnb{halving} of the marked graph \ts $(G,e)$ \ts
as a marked graph $(G'',e)$, where \ts $G''=(V,E'')$ \ts
is obtained by subdivision of each edge other than~$e$.
Note that $G''$ is also proper, planar and simple.  Finally,
define \defnb{simplification} of \ts $(G,e)$ \ts as a marked graph
\ts $(G^\circ,e)$ \ts  obtained by first doubling and then halving.
Note that $G^\circ$ is again proper, planar and simple.


\smallskip

\begin{lemma}\label{lem:halving}
Let  \ts $G=(V,E)$ \ts be a proper planar simple graph, and let \ts $e\in E$.
Then:
\[ \br(G',e) \, = \,  2\ts \br(G,e)\., \quad \br(G'',e) \, = \,  \frac{\br(G,e)}{2}
\quad \text{and} \quad  \br(G^\circ,e) \, = \,  \br(G,e).
\]
\end{lemma}

\smallskip

\begin{proof}
For the first equality, we have:
$$
\tau(G'-e) \, = \, 2^{|V|-1} \. \tau(G-e) \quad \text{and} \quad \tau(G'/e) \, = \, 2^{|V|-2} \. \tau(G/e),
$$
as desired.  Now, note that \ts $(G')^\ast = G''$.  Thus the second equality
 follows from the first and \eqref{eq:planar-dual}.  The third equality follows from the
 first two.
\end{proof}
	
\medskip

\section{Proofs of results}\label{s:proof-main}

\subsection{Proof of Theorem~\ref{thm:graph-sum} }
Let \ts $q_1,\ldots, q_k \in \qqq_{>0}$ \ts be positive rational numbers as in the theorem.
It follows from Proposition~\ref{prop:graph-CF}, that there exists
(not necessarily simple) planar marked graphs \. $(G_i,e_i)$, where \. $G_i = (V_i, E_i)$,
such that
 \begin{align*}
 	\br(G_i,e_i) \, = \,  q_i \quad \text{ and } \quad |E_i| \, = \, \aS(q_i) \ts + \ts 1,
 \end{align*}
for all \ts $1\le i\le k$.  Let \. $(G,e):= (G_1,e_1) \ts \oplus  \ts \cdots \ts \oplus \ts (G_k,e_k)$, where \. $G=(V,E)$.
 It follows from  Lemma~\ref{lem:marked-sum}, that
\begin{align*}
  \br(G,e) \ &= \  \br(G_1,e_1) \. + \. \ldots \. + \. \br(G_k,e_k)\ = \ q_1 \. + \. \ldots \. + \. q_k \qquad \text{and}\\
  |E| \ &= \ |E_1| \. + \. \ldots \. + \. |E_k| \, - \, k \ts + \ts 1 \ = \  \aS(q_1) \. + \. \ldots\. + \. \aS(q_k)\ts +\ts 1.
\end{align*}
Since the marked graph \. $(G,e)$ \. is not necessarily simple, take its simplification \ts
$(G^\circ,e)$, where \ts $G^\circ = (V^\circ,E^\circ)$.
It then follows from Lemma~\ref{lem:halving} that \ts $(G^\circ,e)$ \ts
is a simple planar marked graph satisfying \. $\br(G^\circ,e) = \br(G,e)$ \.
and
$$
|E^\circ| \, = \, 4(|E|-1) \. + \. 1 \, = \, 4\bigl(\aS(q_1) \ts + \ts \ldots\ts + \ts \aS(q_k)\bigr) \ts +\ts 1.
$$
 This completes the proof. \qed

 \smallskip

 \subsection{Proof of Theorem~\ref{t:main}}
Let \. $\ad:=\at-\ac$. Using \eqref{eq:ratio-to-current}, it suffices to show
that there exists a simple planar graph \ts $G=(V,E)$ \ts and edge \ts $e\in E$,
such that for sufficiently large $\at$ we have:
\begin{equation}\label{eq:Kir-proof}
	\br(G,e) \ = \  \frac{\ad}{\ac} \qquad \text{ and } \qquad |E| \ \leq \ C \.
\max \bigg \{\. \frac{\ad}{\ac}\., \, \frac{\ac}{\ad}\., \,\log (\ac+\ad)  \bigg\},
\end{equation}
for some universal constant $C>0$. We prove \eqref{eq:Kir-proof} in the following series of lemmas.
Denote by \ts $C_0>0$ \ts the universal constant in Theorem~\ref{thm:Bou}.

\smallskip

\begin{lemma}\label{lem:Kir-proof-1}
Condition \eqref{eq:Kir-proof} holds for \. $\ad/\ac \. \geq \. \lceil C_0 \log (\ac+\ad)\rceil$.
\end{lemma}

\smallskip

\begin{proof}
Let \. $\ac',\ad'$  \. be coprime positive integers defined as
 \[ \frac{\ad'}{\ac'} \ = \  \frac{\ad}{\ac} \mod 1,
 \]
so we have \. ${\ad'}/{\ac'} \in [0,1)$.
 Note that \. $\ac'+\ad'\leq \ac+\ad$.
 Now, by  Theorem~\ref{thm:Bou},   there exist rational numbers \. $q_1',\ldots, q_k'  \in (-1,1)$ \. such that
 \begin{align}
& \frac{\ad'}{\ac'} \ = \  q_1' \. + \. \ldots \. + \. q_k' \qquad \text{and} \notag  \\
 \label{eq:Kir-0.5}	& \aS\big(|q_1'|\big) \. + \. \ldots \. + \. \aS\big(|q_k'|\big) \ \leq \ C_0 \log (\ac'+\ad') \ \leq \ C_0 \log (\ac+\ad). 
 \end{align}
 Note that it follows from the assumed lower bound on $d/c$ and \eqref{eq:Kir-0.5}
 that
 \begin{equation}\label{eq:k-bound}
 k \ \leq \  C_0 \log(\ac+\ad) \ \leq \  \bigg\lfloor \frac{\ad}{\ac} \bigg\rfloor.
 \end{equation}

 We now define \ts $q_1,\ldots,q_k$ \ts  as
 \[
q_1 \, := \,  \bigg\lfloor	\frac{\ad}{\ac} \bigg\rfloor  -k  + 1 + q_1' \qquad\text{and} \qquad
 q_i \ := \ 1+q_i' \ \  \text{ for all } \ i \geq 2.
 \]
 It follows that \ts $q_1 > 0$ \ts by \eqref{eq:k-bound}, and  \ts $q_2,\ldots, q_k>0$ \ts by definition.
 Also note that
 \begin{equation*}
 	    \frac{\ad}{\ac} \ = \  q_1 \. + \. \ldots \. + \. q_k\..
 \end{equation*}
Observe that for a rational number \. $q=[a_1,\ldots, a_\ell] \in \qqq\cap(0,1)$, we have:
  \begin{align*}
  	1-q \ = \
  	\begin{cases}
  		[1,a_1-1,a_2,\ldots, a_\ell] & \text{ if } a_1 >1,\\
  		[a_2+1,\ldots, a_\ell] & \text{ if } a_1=1.
  	\end{cases}
  \end{align*}
  This implies that
  \[ \aS(1+q_i') \  \leq \ 1+ \aS(|q_i|') \quad \text{for all} \quad 1\le i \le k.
   \]
  Therefore, we have:
  \begin{equation*}
  	   \aS\big(q_1\big)  \. + \. \ldots \. + \.  \aS\big(q_k\big) \ = \
   \bigg\lfloor	\frac{\ad}{\ac} \bigg\rfloor \. - \. k  \. + \.   \aS(1+q_1') \. + \. \ldots \. + \.   \aS(1+q_k') \ \leq \   \. \frac{\ad}{\ac} \. + \.  C_0 \ts \log (\ac+\ad),
  \end{equation*}
where the  inequality is due to  \eqref{eq:Kir-0.5}.

By Theorem~\ref{thm:graph-sum}, there exists a simple planar graph \ts $G=(V,E)$ \ts and edge \ts $e\in E$,
such that
\begin{align*}
	& \br(G,e) \ = \  q_1\. + \. \ldots \. + \.  q_{k}  \ = \   \. \frac{\ad}{\ac} \qquad \text{and} \\
	& |E| \  = \ 4\ts \big(\aS(q_1) \. + \. \ldots \. + \. \aS(q_k)\big) \. + \. 1
	\ \leq \ 4 \. \frac{\ad}{\ac} \. + \.  4\. C_0 \ts \log (\ac+\ad) \. + \. 1.
\end{align*}
Thus \. $(G,e)$ \. satisfies \eqref{eq:Kir-proof} with a constant
\[  C_1 \ := \  4 \. C_0 \. + \. 5.
\]
This  completes the proof of the lemma.
\end{proof}

\smallskip

\begin{lemma}\label{lem:Kir-proof-2}
Condition \eqref{eq:Kir-proof} holds for \. $\ad/\ac \ts \geq 1$.
\end{lemma}

\smallskip

\begin{proof}
Let \. $K:= 4 \ts \lceil  \. C_1 \. \log (\ac+\ad) \rceil$\..
Let \. $L \in \{0,\ldots,K-1\}$ \. be the unique  integer such that
\begin{equation*}
	\frac{\ad}{\ac} \. - \. \frac{L}{K} \ \in \ \big[{1}/{K}, {2}/{K}\big)  \mod 1.
\end{equation*}
Let \. $\ac',\ad'$  \. be coprime positive integers defined as
\begin{equation}\label{eq:dc-0}
 \frac{\ad'}{\ac'} \ = \  \frac{\ad}{\ac} \. - \.  \frac{L}{K} \mod 1,
\end{equation}
so we have
\begin{equation}\label{eq:dc1}
\frac{\ad'}{\ac'} \,\in \,\big[1/K, 2/K\big)\ts.
\end{equation}
Note that
\begin{equation*}
 \log(\ac'+\ad') \, \leq   \, \log(\ac+\ad) \. + \. \log  K  \, \leq \, 2 \ts \log(\ac+\ad),
\end{equation*}
for sufficiently large \. $(\ac+\ad)$.
Note  also that
\begin{equation*}
	\frac{\ac'}{\ad'} \ \geq_{\eqref{eq:dc1}} \ \frac{K}{2} \ = \ 2 \. \lceil C_1 \log (\ac+\ad) \rceil \ \geq \ C_1 \. \lceil \log(\ac'+\ad')\rceil.
\end{equation*}

By Lemma~\ref{lem:Kir-proof-1} applied to \ts $\ac'/\ad'$, there exists a simple
planar graph \ts $G=(V,E)$ \ts and an edge \ts $e \in E$, satisfying
\begin{align*}
& \br(G,e) \ = \  \frac{\ac'}{\ad'} \qquad \text{ and } \\
 & |E| \ \leq \ C_1 \.
 \left (\ts \frac{\ac'}{\ad'} \.  + \. \log (\ac'+\ad') \ts  \right) \  \leq_{\eqref{eq:dc1}} \ C_1 \.
 \big ( K \.  + \. \log (\ac'+\ad')  \big) \leq \  C_1 \. (4 \ts C_1+2) \. \log(\ac+\ad).
\end{align*}
Let \ts $G^\ast=(F,E)$ \ts be the plane dual of \ts $G$.  The marked graph \ts $(G^\ast,e)$ \ts satisfies
\begin{align*}
	\br(G^\ast,e) \ = \  \frac{\ad'}{\ac'} \qquad \text{ and } \qquad |E| \ \leq  \  C_1 \. (4 \. C_1+2) \. \log(\ac+\ad).
\end{align*}
Note that \. $G^\ast$ \. is not necessarily simple.

In a different direction, we have \. $\ad/\ac \ge 1 \ge \ad'/\ac'$ \. by assumption in the lemma.
Also note that \. $K(d/c- d'/c')$ \. is an integer by \eqref{eq:dc-0}.
By Theorem~\ref{thm:graph-sum} applied to \. $q_1\gets (d/c - d'/c')$ \ts and \ts $k=1$,
there exists a planar graph \ts $G'=(V',E')$ \ts and an edge \ts $e' \in E'$, such that
\begin{align*}
	\br(G',e') \ &= \  \frac{\ad}{\ac} \. - \. \frac{\ad'}{\ac'} \quad	\text{and} \quad
|E'| \ \leq \  4 \. \aS\left(\frac{\ad}{\ac} \. - \. \frac{\ad'}{\ac'}\right) \. + \. 1\..\end{align*}
Note that
\begin{align*}
		|E'| \ & \leq \  4 \. \aS\bigg(\frac{\ad}{\ac} -\frac{\ad'}{\ac'}\bigg) \. + \. 1 \  \leq \ 4 \. \bigg(\frac{\ad}{\ac} \. + \. K \bigg) \. + \. 1 \   \leq \ 4 \. \bigg(\frac{\ad}{\ac} \. + \.  4 \. C_1 \. \log (\ac+\ad)\bigg) \. + \. 17,
\end{align*}
where the second inequality is because \. $K(\ad/\ac- \ad'/\ac')$ \. is an integer.

Now, let \ts $(G^\circ,f)$, where \ts $G^\circ=(V^\circ,E^\circ)$ \ts and \ts $f\in E^{\circ}$,
be the marked graph obtained by taking the marked sum \. $(G^\ast,e) \. \oplus \. (G',e')$ \.
followed by simplification. By construction,  $H$ is a simple planar  graph.
It follows from Lemma~\ref{lem:marked-sum} and Lemma~\ref{lem:halving}, that
\begin{align*}
	\br(G^\circ,f) \ &= \  \br(G^\ast,e) \. + \. \br(G',e')  \ = \  \frac{\ad}{\ac} \qquad \text{and}\\
	|E^\circ| \ &\leq \ 4 \ts\big(|E| \ts + \ts |E'| \ts - 2\big) \. + \. 1
\ \leq \  16 \. \frac{\ad}{\ac} \. + \. 8 \ts C_1 \ts (2 \ts C_1 \ts + \ts 9) \. \log (\ac+\ad) \.  + \.  61\ts.
\end{align*}
Thus \. $(G^\circ,f)$ \. satisfies \eqref{eq:Kir-proof} with a constant
$$
C_2 \ := \  16 \ts C_1^2 \. + \.  72\ts C_1 \. + \.  77\ts.
$$
This  completes the proof of the lemma.
\end{proof}

\smallskip

\begin{lemma}\label{lem:Kir-proof-3}
Condition \eqref{eq:Kir-proof} holds for \. $\ad/\ac \. < \. 1$.
\end{lemma}

\begin{proof}
	Since $\ac/\ad > 1$,  applying Lemma~\ref{lem:Kir-proof-2} to $\ac/\ad$ gives  a simple planar marked graph $(G,E)$ satisfying
\begin{align*}
	\br(G,e) \  = \  \frac{\ac}{\ad} \quad \text{and} \quad |E| \ \leq \ C_2  \. \max \bigg \{\.   \frac{\ac}{\ad}\., \,\log (\ac+\ad)  \bigg\}.
\end{align*}
Apply the planar dual and then the simplification operations to the marked graph
\ts $(G,e)$ \ts to obtain a simple planar marked graph \ts $(G^\circ,e)$,
where \ts $G^\circ=(V^\circ,E^\circ)$ \ts and \ts $e\in E^\circ$.  We have:
\begin{align*}
	& \br(G^\circ,e) \ = \ \frac{1}{\br(G,e)} \ = \ \frac{\ad}{\ac} \qquad \text{and} \\
	& |E^\circ| \ = \  \ 4 \. \big(|E| \ts - \ts 1) \. + \. 1 \  \leq  \ 4\ts |E| \  \leq  \  4 \.  C_2  \. \max \bigg \{\.   \frac{\ac}{\ad}\., \,\log (\ac+\ad)  \bigg\}.
\end{align*}
Thus \. $(G^\circ,e)$ \. satisfies \eqref{eq:Kir-proof} with a constant \. $4\ts C_2$\ts.
This completes the proof of the lemma and finishes the proof of the theorem.
\end{proof}

\smallskip

\subsection{Proof of Proposition~\ref{p:main-lower}}\label{ss:lower}
The proposition consists of three inequalities which we prove separately.
For the first and second inequality, we need the following result:

\smallskip

\begin{lemma}  \label{l:commute}
For every connected graph \ts $G=(V,E)$ \ts and edge \ts $e=(x,y)\in E$,
we have:
$$
\bI(G,e) \, \ge \, \frac{1}{2} \left(\frac{1}{\deg(x)} \. + \. \frac{1}{\deg(y)}\right),
$$
where \ts $\deg(v)$ \ts denotes the degree of vertex \ts $v\in V$.
\end{lemma}
\smallskip

The lemma is well known via connection to the random walks on~$G$ and the
\emph{commute time} \ts interpretation 
$$\kappa(G,e) \, = \, 2\ts |E|  \bI(G,e)
$$
given in \cite[Thm~2.1]{C+96}, combined with the inequality  
$$\kappa(G,e) \, \ge \,  |E| \.  \left(\frac{1}{\deg(x)} \. + \. \frac{1}{\deg(y)}\right), 
$$
see e.g.\ \cite[Cor.~3.3]{Lov-survey}.  Since \ts $G$ \ts is simple, we have \ts $\deg(x),\deg(y)<|V|$.
This implies the first inequality:
\begin{equation}\label{eq:lower-part1}
|V| \, > \, \frac{1}{\bI(G,e)} \, = \, \frac{\at}{\ac}\..
\end{equation}

For the second inequality, denote by $F$ the set of faces in~$G$.
Since \ts $G$ \ts is simple and planar, it follows from Euler's formula
that \ts $|E| \le 3 \ts |V|-6$ \ts and \ts $|F| \le 2 \ts |V|-4$.
Applying Lemma~\ref{l:commute} to the dual graph \ts $G^\ast=(F,E)$, we conclude:
\begin{equation}\label{eq:lower-part2}
|V|\, > \, \frac{|F|}{2} \, >_{\eqref{eq:lower-part1}} \, \frac{1}{2 \bI(G^\ast,e)} \,
=_{\eqref{eq:planar-dual}} \,
\frac{1}{2\bigl(1-\rho(G,e)\bigr)} \, = \, \frac{1}{2}\left(\frac{\at}{\at-\ac}\right).
\end{equation}
Finally, for the third inequality, we have:
\begin{equation}\label{eq:lower-part3}
\at \, \le \, \tau(G) \, < \, 2^{|E|} \, < \, 2^{3\ts |V|}\.,
\end{equation}
as desired.  \qed

%

%

\medskip

\section{Final remarks and open problems}\label{s:finrem}

\subsection{} \label{ss:finrem-Hall}
Bourgain's Theorem~\ref{thm:Bou} was partially motivated by
Hall's classic result that every number in the interval \ts
$(\sqrt{2}-1,4\sqrt{2}-4)$ \ts can be presented as
the sum of two continued fractions whose partial quotients
do not exceed four \cite[Thm~3.1]{Hall47}.  Han\v{c}l and
Turek \cite{HT23} gave a version of this result for partial
fractions of the type $[a_1,1,a_2,1,a_3,1,\ldots]$\ts.
In our previous paper \cite{CKP}, we used partial fractions
of this type to study Sedl\'a\v{c}ek's problem, see~$\S$\ref{ss:intro-Sedlacek}.
While initially, we intended to obtain Theorem~\ref{t:main}
via a finite version of the Han\v{c}l--Turek result,
this turned out to be unnecessary due to the
simplification operation given in~$\S$\ref{ss:halving}.

\subsection{} \label{ss:finrem-LE}
For a finite poset \ts$P=(X,\prec)$, the \emph{relative number
of linear extensions} \. $e(P-x)/e(P)$ \ts plays roughly the
same role as the effective resistance for graphs.  In \cite{CP-CF},
we used recent analytic estimates on sums of quotients of continued
fractions to show that for all \ts $d/3\ge c\ge 1$ \ts there exists a poset
\ts $P=(X,\prec)$ \ts and an element \ts $x\in X$, such that
$$
\frac{e(P-x)}{e(P)} \, = \, \frac{c}{d} \quad \text{and} \quad |X| \. \le \. C \ts
\max\left\{\ts \frac{d}{c} \. , \, \log d \. \log \log d \ts\right\}.
$$
By analogy with Theorem~\ref{t:main}, one can hope to remove the \ts $\log \log d$ \ts
factor; this was in fact conjectured in \cite[Eq.~(1.7)]{CP-CF}.  Unfortunately,
the approach in this paper is not applicable since the analogue of the marked sum
for posets called the \emph{flip-flop construction} \ts defined in \cite[$\S$3.3]{CP-CF}
cannot be used more than once.

\subsection{} \label{ss:finrem-const}
Throughout the paper we made no effort to compute or optimize the constants.
This is
in part
because the constant $C$ in Bourgain's Theorem~\ref{thm:Bou} is
not specified.  We note, however, that the constant~$C'$ in Proposition~\ref{p:main-lower}
can be easily computed.  In fact, a more careful argument shows that one
can remove the constant $1/2$ in the inequality \eqref{eq:lower-part2}.
Similarly, one can improve \eqref{eq:lower-part3} by using the
\ts $\tau(G) < 5.23^{|V|}$ \ts bound in \cite{BS10}.  After these
improvements, one can take \ts $C':=0.6$.

\vskip.6cm
{\small

\subsection*{Acknowledgements}
We are grateful to Noga Alon, Matija Buci\'c, Peter Sarnak, Ilya Shkredov
and Nikita Shulga for helpful comments and remarks on the literature.
SHC~was supported by NSF grant DMS-2246845.  AK~was supported by
NSF grant DMS-2302641, BSF grant 2020119 and a Simons Fellowship.
IP~was supported by NSF grant CCF-2302173.  This paper
was written when AK was visiting the Institute for Advanced Study.
We are grateful for the hospitality.
}

\vskip1.1cm


{\footnotesize

\vskip.6cm
}

\end{document}